\documentclass{amsart}

\usepackage[utf8]{inputenc}
\usepackage[T1]{fontenc}
\usepackage[english]{babel}
\usepackage{color}
\usepackage{amsmath} 
\usepackage{amssymb} 
\usepackage{amsthm}
\usepackage{graphicx}
\usepackage[colorlinks=true, linkcolor=blue]{hyperref}
\usepackage{cancel}

\usepackage{lmodern}
\usepackage{microtype}

\theoremstyle{plain}
\newtheorem{thm}{Theorem}[section]
\newtheorem*{thm*}{Theorem}
\newtheorem{prop}{Proposition}[section] 
\newtheorem{lem}{Lemma}[section] 
\newtheorem{cor}{Corollary}[section]
\newtheorem*{cor*}{Corollary}

\newcommand {\R} {\mathbb{R}} \newcommand {\Z} {\mathbb{Z}}
\newcommand {\T} {\mathbb{T}} \newcommand {\N} {\mathbb{N}}

\newcommand {\p} {\partial}
\newcommand {\dt} {\partial_t}

\usepackage{enumitem}

\DeclareMathOperator{\supp}{supp}
\DeclareMathOperator{\rhs}{rhs}
\begin{document}
\title[Plasma Echo Chains]{On Echo Chains in Landau damping: Self-similar Solutions and Gevrey 3 as a
  Linear Stability Threshold}
\author{Christian Zillinger}
\address{BCAM -- Basque  Center  for  Applied  Mathematics, Mazarredo 14, E48009
  Bilbao, Basque Country -- Spain}
\email{czillinger@bcamath.org}
\begin{abstract}
  We show that the linearized Vlasov-Poisson equations
  around self-similar non-homogeneous states near zero contain the full plasma
  echo mechanism, yielding Gevrey 3 as a critical stability class.
  Moreover, here Landau damping may persist despite blow-up:
  We construct a critical Gevrey regularity class in which the force
  field converges in $L^2$. Thus, on the one hand, the physical phenomenon of
  Landau damping holds. On the other hand, the density diverges to infinity in
  Sobolev regularity. Hence, ``strong damping'' cannot hold.
\end{abstract}
\maketitle
\tableofcontents

\section{Introduction}
\label{sec:intro}
In this article we are interested in the echo mechanism for the Vlasov-Poisson
equations
\begin{align}
  \label{eq:1}
  \dt f + (v, F) \cdot \nabla_{x,v} f=0,
\end{align}
where
\begin{align*}
  F(t,x)=  \nabla_x W * \rho,  \ \rho= \int f dv, \ |\hat{W}(k)|=|k|^{1-s},
\end{align*}
is the force field generated by the density. Here, the physically most
interesting cases are given by gravitational and Coulomb interaction for which
$F(t,x)= \pm \nabla \Delta^{-1}\rho$.
It is a classical result going back to Landau \cite{Landau} that under suitable conditions smooth homogeneous
solutions $f=f_0(v)$ are linearly asymptotically stable. More precisely, in the
linearized problem with Sobolev regular initial data the perturbation of $f_0$
asymptotically converges to a solution of the free transport problem and as a
result the perturbation of the force field $F$ tends to zero as $t\rightarrow
\infty$. This decay of the force field is known as (linear) \emph{Landau
  damping} and continues to be of great importance in the study of plasma physics.

Following the results on the linearized problem it was a natural longstanding
question whether and under which conditions nonlinear Landau damping may hold.
Here, in a seminal work Mouhot and Villani \cite{Villani_long} showed that for
Gevrey regular $f_0$ satisfying the Penrose condition and sufficiently small
Gevrey regular perturbations nonlinear Landau damping holds. In that work the
high regularity requirement, which differs strongly from the mild requirements
in the linearized problem, is introduced due to worst case growth estimates in a
``toy model''. In \cite{bedrossian2016nonlinear} Bedrossian showed that there
indeed exists initial data attaining norm inflation as in the toy model.
However, we note that such a result does \emph{not} rule out Landau damping.
More precisely, we note there is a \emph{hierarchy} of increasingly stronger statements
established in the mathematical and physical literature which are each referred to
as damping:
\begin{enumerate}[label={[\arabic*]}]
\item \label{item:physical} The force field converges to a (simpler) asymptotic profile as time tends
  to infinity. This is the physically observed phenomenon. 
\item The corresponding perturbations to the phase-space density remains
  uniformly bounded in a suitable $L^p$ space and asymptotically converges
  weakly as time tends to infinity. 
\item \label{item:strong} The perturbation asymptotically behaves like a free solution of an
  associated linear problem. In the language of dispersive equations one says
  that the solution scatters with respect to the linear dynamics.
\end{enumerate}
The aim of the present article is two-fold:
\begin{itemize}
\item We propose that the Gevrey regularity requirement should be understood as
  a \emph{secondary linear instability}. That is, in any small analytic neighborhood of
  $f_0(v)$ there exist \emph{self-similar solutions}
  \begin{align*}
    f(t,x,v)=f_0(v)+ f_{s}(x-tv,v), t>\epsilon>0,
  \end{align*}
  whose \emph{linearized} problem exhibits full \emph{echo chains} and is unstable in
  sub Gevrey-3 regularity. Throughout this article we consider the special case
  $f_0(v)=0$. As we remark after Theorem \ref{thm:main}, while many results of
  Section \ref{sec:better} also hold for more general $f_0$, our blow-up results
  exploit separation of supports, while $f_0\neq 0$ might introduce (small) overlaps. 
\item Even if ``strong damping'' in the sense \ref{item:strong} fails, this does not
  necessarily imply that ``physical Landau damping'' in the sense \ref{item:physical} fails.
  In this work we show that for the linearized problem around self-similar
  solutions this scenario indeed happens. That is, we construct \emph{Gevrey regular initial
    data}, such that $f(t)$ \emph{diverges} in Sobolev regularity (norm
  inflation to infinity) but such that the force field $F(t)$ still
  \emph{converges} in $L^2$ as $t\rightarrow \infty$.
\end{itemize}
\subsubsection*{Main Results}
\label{sec:results}

We show in Section \ref{sec:selfsimilar} that the Vlasov-Poisson equations
possess numerous self-similar solutions of the form
\begin{align}
  \label{eq:6}
  f_*(t,x,v)=f_0(v)+ \epsilon \cos(x-tv) \psi(v), t>T>0.
\end{align}
where $\psi(v)$ is compactly supported in Fourier space inside a ball
$B_{\delta}(0)$. Such densities are small analytic perturbations of the
homogeneous solution $f_0(v)$ and solve the Vlasov-Poisson equations on
$(\delta,\infty)$.

We remark that Bedrossian \cite{bedrossian2016nonlinear} considered a similar
ansatz in his work on nonlinear echo chains, where however the Fourier transform
$\hat{\psi}(\eta)=\exp(-|\eta|)$ was not compactly supported and thus resulted
in a non-self-similar solution. The support restriction not only allows us to
explicitly determine a solution to the nonlinear Vlasov-Poisson equations to
perturb around but also allows us to cleanly isolate the echo chain mechanism in
the linearized problem for arbitrarily large frequencies.

Our main results are obtained in Theorems \ref{thm:inflation}, \ref{thm:blowup}
and \ref{thm:stability} and summarized in the following theorem.

\begin{thm}
  \label{thm:main}
  Let $\psi \in \mathcal{S}(\R)$ with $\text{supp}(\hat{\psi})\subset (-\delta,\delta)$, $0<\delta<0.1$. Then
  \begin{align*}
    f_*(t,x,v)=f_0(v)+ \epsilon \cos(x-tv) \psi(v)
  \end{align*}
  is a solution of the Vlasov-Poisson equations on $(0.1,\infty)$.\\

  Consider the linearized Vlasov-Poisson equations around $f_*$ with $f_0\equiv 0$
  on $(0,\infty)$ with initial data $h_0$:
  \begin{align*}
    \tag{LVP}
    \begin{split}
    \dt h + F[\int h(t,x-tw,w)dw] (\nabla_{v}-t \nabla_x)(\epsilon \cos(x_1) \psi(v))&=0,\\
    \hat{F}[\rho](k)&= k\hat{W}(k) \hat{\rho}(k), \\
    h(0)&=h_0,
    \end{split}
  \end{align*}
  with $|\hat{W}(k)|=|k|^{-2}$.
  \begin{itemize}
  \item If $h_0 \in \mathcal{G}_{3}$ is Gevrey regular with a sufficiently large
    constant $c>0$,
    \begin{align*}
      \sum \int |\tilde{h}_0|^2 \exp(c \sqrt[3]{|\eta|}) d\eta\leq C_{0}<\infty,
    \end{align*}
    then there exists a constant $C>0$ (independent of $h_0$ or $C_0$) such that
    for all times
     \begin{align*}
    \sum_{k}\int |\tilde{h}(t,k,\eta)|^2 \exp(\sqrt[3]{|\eta|}) d\eta \leq C C_{0}.
     \end{align*}
     The linearized problem is \emph{stable in Gevrey regularity} and linear Landau damping
     holds.
   \item Suppose in addition that $\hat{\psi}\geq 0$, then there exists initial
     data $h_0 \in \mathcal{G}_3$, supported in frequency in 
     $\{k_0\} \times (\eta_0-1/2,\eta_0+1/2)$ such that the solution
     $h(t)$ with this initial data is stationary for $t>\eta_0+1/2+ \delta
     k_0=:T_1'$ and there exist constants $c_1,c_2$ (proportional to
     $\|\psi(v)\|_{L^\infty}$ and independent of $\eta_0$) such that
     \begin{align*}
       \exp(\sqrt[3]{c_1 \eta_0})\leq   \|h(T_1')\|_{L^2} \leq \exp(\sqrt[3]{c_2 \eta_0}).
     \end{align*}
     There is \emph{norm inflation} due to \emph{echo chains}.
  \item For every $s \in \R$ there exists $h_\infty \in H^{s}\setminus H^{s+}$ and
  $h_0 \in \mathcal{G}_{3}$ (with small constant) such that the solution $h(t)$ of \eqref{eq:LVP} with
  initial datum $h_0$ converges to $h_\infty$ in $H^s$ as $t \rightarrow
  \infty$. In particular, the solution $h(t)$ \emph{diverges} in $H^{s+}$. However, if $s\geq
  0$ then $F[h](t)\rightarrow_{L^2} 0$ as $t\rightarrow \infty$.
  For this data physical linear damping in the sense \ref{item:physical} holds, but strong
  damping to transport in the sense \ref{item:strong} fails.
  \end{itemize}
\end{thm}
We thus show that indeed (linear) Landau damping in the sense \ref{item:physical} may persist despite blow-up,
that the norm inflation mechanism is (secondary) linear and that it may not only result in
inflation but blow-up. 
Here, the choice of $f_0\equiv 0$ serves to preserve the Fourier support (of
some modes) under the evolution. Similarly to results of the author and Deng in
fluid mechanics \cite{dengZ2019} we expect that an extension to more general
$f_0$ and perturbations satisfying a smallness condition such
as in \cite{bedrossian2016nonlinear}, where $(2.10)$ reads 
\begin{align*}
  \epsilon \langle k_0,\eta_0\rangle^{-R} \exp(3 (K_m'\epsilon \eta_0)^{1/3})=1,
\end{align*}
is possible with some effort. However, the above blow-up result crucially relies
on considering sequences of $\eta_0$ tending to infinity. Here, in the setting
of inviscid damping the dynamics differ qualitatively \cite{dengZ2019} and we thus
expect that also in the present setting such an extension would require several
new techniques.\\

It would be very interesting to know whether the results of Theorem \ref{thm:main} or similar behaviors persist also in the
nonlinear problem. Here, the case of a single echo chain (around a different not
self-similar solution and at bounded frequency) has been established by
Bedrossian in \cite{bedrossian2016nonlinear}. However, in order to better
understand the possible behavior in lower regularity it remains a challenge to
understand the problem of echo chains at higher frequencies and the interaction
of (countably many) chains.

For an in-depth discussion of the history and the derivation of the Vlasov-Poisson
equations and an overview of the literature we refer to the seminal works of
Villani and Mouhot \cite{Villani_script, Villani_long} and of Bedrossian,
Masmoudi and Mouhot \cite{bedrossian2016landau}.

The remainder of our article is structured as follows:
\begin{itemize}
\item In Section \ref{sec:selfsimilar} we introduce families of self-similar
  non-stationary solutions of the Vlasov-Poisson equations in any dimension as
  well as the specific solutions in $1+1$ dimensions which we study in the
  following.
\item In Section \ref{sec:chains} we briefly recall the underlying physical echo
  mechanism. Furthermore, we introduce a toy model to estimate the possible norm
  inflation due to a chain of echoes.
\item In Section \ref{sec:better} we introduce the linearized problem \eqref{eq:LVP} and study
  its evolution along a chain of echoes. Here we use the compact frequency
  support of the underlying solution $f_*$ to clearly separate times of echoes and
  construct the chain by an iteration.
\item In Section \ref{sec:blow-up} we construct initial data which exhibits norm
  inflation in Gevrey $3$ regularity due to a single echo chain of arbitrarily
  long length. Furthermore, we construct initial data in Gevrey $3$ regularity
  which exhibits infinitely many echo chains and \emph{diverges} in Sobolev
  regularity, but whose force field \emph{converges} in $L^2$.
\item In Section \ref{sec:stability} we complement the norm inflation and
  blow-up result by a stability result in the Gevrey class $\mathcal{G}_3$ with
  large constant.
\end{itemize}

\subsubsection*{Acknowledgments}
Christian Zillinger's research is supported by the ERCEA under the grant 014
669689-HADE and also by the Basque Government through the BERC 2014-2017 program
and by Spanish Ministry of Economy and Competitiveness MINECO: BCAM Severo Ochoa
excellence accreditation SEV-2013-0323.

\section{Self-similar Non-stationary States}
\label{sec:selfsimilar}

It is a classical result that the linearized problem around homogeneous states
$f^0=f^0(v)$ reduces to a Volterra equation for the Fourier transform of the
density $\rho=\int f dv$ and for analytic data can be solved by employing
Laplace transforms \cite{Villani_script}. Furthermore, it turns out that this
linear problem is asymptotically stable also in Sobolev regularity. In contrast
the nonlinear problem is stable in Gevrey 3 regularity and unstable in lower
regularity \cite{bedrossian2016nonlinear, bedrossian2013landau}.

We propose that an intuitive explanation of this dichotomy is given by the
existence of a large family of nearby \emph{self-similar non-stationary states}
\begin{align}
  \label{eq:3}
  f_*(t,x,v)= \epsilon g(x-tv) h(v) + f^0(v)
\end{align}
for $t>T>0$ which are initially arbitrarily close to homogeneous states (in
analytic regularity) but for which the linearized problem is stable in the
Gevrey 3 class and asymptotically unstable in any weaker regularity class.\\

We remark that any function $f_*$ of the structure \eqref{eq:3} is a solution of
the free transport equations. Thus, in order to also be a solution of the
Vlasov-Poisson equations we need to show that the corresponding force field
vanishes. Since $F$ depends on $f_*$ only in terms of $\rho=\int f_* dv$, a
sufficient condition is given by requiring $g$ and $h$ to be compactly supported in
Fourier space.

\begin{lem}
  \label{lem:trivial}
  Let $0\leq T_1 \leq T_2$ and define the following set in Fourier space:
  \begin{align*}
    \Omega(T_1,T_2)=\left\{(k,\eta)\in \Z^{d} \times \R^d: \exists t \in [T_1,T_2] \text{ s.t. } \eta-kt=0\right\}.
  \end{align*}
  If $f_0 \in L^1$ satisfies $\tilde{f}_0=0$ on $\Omega(T_1,T_2)$, then
  \begin{align*}
    f(t,x,v)=f_0(x-tv,v)
  \end{align*}
  is a self-similar solution of the Vlasov-Poisson equations \eqref{eq:1} on the
  time interval $(T_1,T_2)$.
\end{lem}

\begin{proof}
  We note that by construction $f$ satisfies
  \begin{align*}
    \dt f + v \p_x f =0.
  \end{align*}
  Thus we only need to show that $F[\int f dv]$ identically vanishes. We observe
  that for any $k \in \Z^{d}$ it holds that
  \begin{align*}
    \int e^{-ikx} \int f(t) dv dx = \tilde{f_0}(k,kt)=0,
  \end{align*}
  since $(k,kt) \in \Omega(T_1,T_2)$. Therefore $\int f dv=0$, which implies the
  result.
\end{proof}

The two main examples of interest are given by a single wave packet solution
\begin{align}
  f_0(t,x,v)= \epsilon \sin(e_1\cdot (x-tv)) \psi(v)
\end{align}
and by its perturbation by another wave packet
\begin{align}
  g(t,x,v)= f_0(t,x,v) + \epsilon^2  \sin(k_0\cdot (x-tv)) \sin(\eta_0 \cdot v) \psi(v).
\end{align}
Here $\psi$ is a function which is compactly supported in Fourier space inside a
ball $B_R(0)$. As we show in the following corollary, the function $f_0$ is a self-similar
non-stationary solution of the Vlasov-Poisson equations on $(\delta,\infty)$.
The function $g$ is a small high-frequency perturbation of this solution and is
a solution of the Vlasov-Poisson equations until a time $T=T(k_0,\eta_0,\delta)$
at which the perturbation becomes resonant and results in a perturbation of
the force field. Our main aim in Sections \ref{sec:better} and \ref{sec:blow-up}
is to show that this resonance actually causes a \emph{chain of echoes} and
norm inflation in Gevrey 3 regularity.

\begin{cor}
  \label{cor:trivial}
  Let $\psi(v)$ be compactly supported in Fourier space inside a ball $B_R$
  around zero. Then the function
  \begin{align*}
    f_0(t,x,v)= \epsilon \sin(e_1\cdot (x-tv)) \psi(v)
  \end{align*}
  is a self-similar solution of the nonlinear Vlasov-Poisson equations
  \eqref{eq:1} on the time-interval $(R,\infty)$. Furthermore, there exists a
  Gevrey regular solution $f$ of the Vlasov-Poisson equations \eqref{eq:1} on
  $(0,\infty)$ which agrees with $f_*$ on $(R,\infty)$.
\end{cor}

\begin{proof}
  The function $f_0$ is a solution of the Vlasov-Poisson equations on
  $(R,\infty)$ by Lemma \ref{lem:trivial}. We further note that $f_0(R,x,v)$ is
  analytic and may thus use the local well-posedness theory of the
  Vlasov-Poisson equations to solve the equations on $(0,R)$. Combining both
  solutions we obtain a global solution.
\end{proof}

In the following we will consider $R=\delta \ll 1$ and study the linearized
problem around such a function $f_0$ on $(\delta,\infty)$. Let thus
$f=f_0+\epsilon h$ be a solution of the Vlasov-Poisson equations
\begin{align*}
  \dt f + v \cdot \nabla_x f + F \cdot \nabla_v f =0.
\end{align*}
Then we obtain that
\begin{align*}
  \dt h + v \cdot \nabla_x h + F[h] \cdot \nabla_v f_0 = \mathcal{O}(\epsilon).
\end{align*}
In the linearization we now neglect the $\mathcal{O}(\epsilon)$ error term and
further switch to coordinates $(t,x+tv,v)$ moving with free transport:
\begin{align}
  \label{eq:self}
  \begin{split}
    \dt h + F[h](x-tv) \cdot (\nabla_v-t\nabla_x) (\cos(x) \psi(v))=0, \\
    F[h](x)=\nabla_x W * \left( \int h(t,\cdot-tw,w) dw \right).
  \end{split}
\end{align}
We note that in these coordinates the self-similar solution $f_*$ is stationary.
However, unlike in the case of a homogeneous solution $f_0=f_0(v)$ here $t
\nabla_x f_0(x,v)$ is non-trivial and potentially very large for large times.
Furthermore, the multiplication by $\cos(x)$ introduces a coupling between
neighboring frequencies $k-1, k, k+1$ with respect to $x$.

As we will see in Sections \ref{sec:better} and \ref{sec:blow-up} this nearest
neighbor interaction allows us to propagate resonances, which make this problem
unstable in any sub Gevrey 3 regularity class. The underlying physical mechanism
here is given by \emph{chains} of \emph{plasma echoes}, which we introduce in
terms of a model problem in the following Section \ref{sec:chains}.

\section{Plasma Echoes and Echo Chains}
\label{sec:chains}
We note that the Vlasov-Poisson equations \eqref{eq:1}
\begin{align*}
  \dt f + v \cdot \nabla_x f + F \cdot \nabla_v f =0
\end{align*}
can be interpreted as a transport problem on $\T^d\times \R^d$ by the
divergence-free vector field $(F(t,x),v) \in \R^{2d}$. Thus formally all $L^p$
norms of $f$ are conserved and the evolution is invertible. Nevertheless one
observes in experiments that asymptotically the force field $F$ decays in time
very quickly. This asymptotic convergence of the force field to zero is known as
Landau damping after the works of Lev Landau \cite{Landau} who in the 1940s
mathematically discovered this effect and established it for the linearized
Vlasov-Poisson equations around smooth homogeneous (i.e. $x$ independent)
initial data. The seeming contradiction between asymptotic convergence and
invertibility can be resolved by noting that under the free transport dynamics
$f$ weakly converges to its $x$ average (but not strongly) and thus the force
field strongly decays.

However, while the free transport dynamics are a good approximation for some
times, as we have seen in Section \ref{sec:selfsimilar}, this is not the case for
all times. Indeed, suppose that $f(t=0)$ is concentrated at a high frequency
$(k,\eta)$, then
\begin{align*}
  \tilde{f}(t,k,\eta)= \tilde{f}(0,k,\eta-kt)
\end{align*}
will be concentrated near the zero frequency in $v$ around the time
$t=\frac{\eta}{k}$ and hence yield a significant contribution to the force
field. This resonance mechanism also underlies the seminal physical experiments
by Malmberg et al \cite{malmberg1968plasma} in the 1960s. The following sketch
is adapted from the lectures notes of Villani \cite[page 110]{Villani_long}. For
a more in-depth discussion we further refer to Villani and Mouhot's article
\cite{Villani_long} on Landau damping as well as the articles by Bedrossian,
Mouhot and Masmoudi \cite{bedrossian2013landau}, \cite{bedrossian2016nonlinear}.

In the experimental setup a plasma is confined to a cylinder and one measures
its reaction to frequency localized perturbations. For simplicity of notation in
the following we will describe the underlying mechanism in terms of a periodic
box $\T \times\R$ (identifying the angle with the periodic direction). Given a
plasma at rest, at time zero we introduce a frequency localized perturbation
\begin{align*}
  \epsilon e^{ikx}
\end{align*}
to the density $f$, which then evolves by free transport as
\begin{align*}
  \epsilon e^{ik+ikt v}.
\end{align*}
In particular, since the perturbation is moving to higher and higher frequencies
we observe a damping of the perturbation of the force field. At a later time
$\tau$ we add another perturbation localized at the frequency $l-k$
\begin{align*}
  \delta e^{i(l-k)x},
\end{align*}
which again evolves by transport and is damped. While both perturbations do not
interact in the linear problem, the quadratic nonlinearity of the Vlasov-Poisson
equations introduces a correction of the form
\begin{align*}
  \delta \epsilon C e^{ilx + ik\tau y},
\end{align*}
which then evolves as
\begin{align*}
  \delta \epsilon C e^{ilx + (ik \tau + il(t-\tau))v}.
\end{align*}
While this contribution initially is quadratically small and at high frequency,
given a suitable choice of signs of $k$ and $l$ this contribution may unmix and
become resonant at the time $t$ such that
\begin{align*}
  k\tau + l (t-\tau)=0.
\end{align*}
Thus, while both individual perturbations are quickly damped (at the level of
the force field) their nonlinear interaction results in a peak of the force
field at the frequency $l$ at a later time $t$. The perturbations result in a
\emph{plasma echo}.

We note that this resonance is not present in the linearized problem around
$f_0\equiv 0$ (that is, the free transport problem) or the linearization around
homogeneous states $f_0=f_0(v)$. It is thus considered a nonlinear phenomenon
and suggests that the nonlinear dynamics may differ strongly from the linearized
dynamics for large times. In particular, while the linearized problem around
homogeneous states is stable in Sobolev regularity, the nonlinear problem
requires Gevrey 3 regularity \cite{Villani_long}, \cite{bedrossian2013landau},
\cite{bedrossian2016nonlinear}.

However, as the main result of this article we show that echoes are a
\emph{linear effect} when one considers non-homogeneous \emph{self-similar}
solutions. That is, in any arbitrarily small, smooth neighborhood of homogeneous
solutions there are time-dependent self-similar solutions $f_*$ to the nonlinear
Vlasov-Poisson equations whose (linearized) perturbations result in resonances.
Furthermore, this linearized problem does not only exhibit single echoes but
\emph{chains of echoes}, where the first echo excites another mode, which
results in an echo at a later time, which in turn excites another mode etc. In
Section \ref{sec:better} we show that the linearized problem \eqref{eq:self}
with initial data localized near the frequency $(k_0,\eta_0)$ undergoes an echo
chain
\begin{align*}
  k_0 \rightarrow k_0-1 \rightarrow k_0-2 \rightarrow \dots \rightarrow 0
\end{align*}
of length $k_0$ and as a result exhibits norm inflation in Gevrey 3
regularity. Furthermore, combining infinitely many echo chains of increasing
length in Section \ref{sec:blow-up} we construct smooth initial data
which exhibits \emph{blowup} in Sobolev regularity as time tends to infinity but
whose force field \emph{converges} in $L^2$.

\subsection{A Model Problem}
\label{sec:echo}

We are interested in the evolution of the linearized Vlasov-Poisson equations
\eqref{eq:self} around a self-similar solution $f_*$ with $f_0(v)=0$:
\begin{align}
  \dt h + F[\int h(t,X-t\sigma,\sigma) dv](X=x-tv) (\p_v - t\p_x) \cos(x)\psi(v) =0.
\end{align}
We remark that in the linearized problem around homogeneous states there is no
$\cos(x)$ and the equation may be explicitly solved by using Fourier and Laplace
transforms \cite{Villani_script}. As the multiplication by $\cos(x)$ introduces
a coupling between neighboring frequencies resonances may excite nearby modes
and in turn lead to resonances at later times. Here, in contrast to the setting
of fluid dynamics considered in \cite{dengZ2019} the compact support of $\psi$
in frequency and the lower dimensional structure of the map $h \mapsto F[h]$
yields a much simpler control of higher order Duhamel iterates which we use in
Section \ref{sec:better}.\\

In order to introduce ideas and to see what growth to expect we first discuss
the evolution for a \emph{toy model} (see also \cite{Villani_long, bedrossian2016nonlinear} for some other toy models).

We consider an initial perturbation $h_0$ which is concentrated in Fourier space
in $\{k_0\}\times [\eta_0-1/2,\eta_0+1/2]$ for some large $k_0,\eta_0$. If those
have opposite signs (and $|\eta|>1/2$), then by Corollary \ref{cor:trivial} the
function $h$ is a stationary solution (in coordinates moving with the free
transport) and hence trivially stable. We thus assume that both $\eta_0$ and
$k_0$ are positive and $\eta_0$ is much larger than $k_0$. Then by Lemma
\ref{lem:trivial} the function $h$ is a stationary solution until the first
critical time
\begin{align*}
  T_{k_0}: \eta_0-1/2-k_0 T_{k_0}=0.
\end{align*}
Following a similar approach as in Bedrossian's work
\cite{bedrossian2016nonlinear} we consider the Fourier transform of equation
\eqref{eq:self}:
\begin{align}
  \label{eq:self2}
  \dt \tilde{h}(t,k,\eta)= \sum_{l=k\pm 1} l\hat{W}(l) \tilde{f}(t,l,lt) ((\eta-lt)+t \text{sgn}(k-l)) \tilde{\psi}(\eta-lt).
\end{align}
In our toy model we now freeze all modes except the mode $k_0-1$ and thus obtain
that for $T>T_{k_0}$
\begin{align*}
  \tilde{h}(T,k_0-1,\eta)= \int_{T_{k_0}}^T k_0 \hat{W}(k_0) \tilde{f}(0,k_0, k_0t) (1+t) \tilde{\psi}(\eta-k_0t) dt.
\end{align*}
Choosing $T=T_{k_0}'$ such that $\eta_0+1/2-k_0 T_{k_0}'=0$, we formally
approximate
\begin{align*}
  & \quad \int_{T_{k_0}}^T k_0 \hat{W}(k_0) \tilde{h}(0,k_0, k_0t) (1+t) \tilde{\psi}(\eta-k_0t) dt\\
  &\approx k_0 \hat{W}(k_0) \frac{\eta_0}{k_0} \int_{T_{k_0}}^{T} \tilde{h}(0,k_0, k_0t) \tilde{\psi}(\eta-k_0t) dt\\
  &= k_0 \hat{W}(k_0) \frac{\eta_0}{k_0} \frac{1}{k_0} \int_{\eta_0-1/2}^{\eta_0+1/2} \tilde{h}(0,k_0, \tau) \tilde{\psi}(\eta-\tau) d\tau, \\
  &=  \hat{W}(k_0) \frac{\eta_0}{k_0} \int_{\R} \tilde{h}(0,k_0, \tau) \tilde{\psi}(\eta-\tau) d\tau,
\end{align*}
where we used the compact support of $\tilde{f}(0,k_0,\cdot)$ in the last step.

Thus the perturbation $f_0=\epsilon \cos(x) \psi(v)$ which we linearize around
and our initial perturbation $h(t=0)$ near frequency $(k_0,\eta_0)$ results in
an \emph{echo} at frequency $k_0-1$ during the time $(T_{k_0}, T_{k_0}')$. We
note that here we obtained several factors:
\begin{itemize}
\item $k_0 \hat{W}(k_0)$ is determined by the force modeled, which is
  proportional to $k_0^{-1}$ in the case of gravitational or Coulomb
  interaction.
\item $\frac{\eta_0}{k_0}$ is an approximation of $(1+t)$ near the resonant
  time.
\item $\frac{1}{k_0}$ accounts for the change of variables in time or
  equivalently the speed with which we travel through the support of
  $\tilde{\psi}$.
\end{itemize}
If we now further consider $\psi$ as $\sigma$ times an approximate identity our
toy model becomes
\begin{align}
  \tilde{h}(T_{k_0}',k_0-1,\eta) \approx \hat{W}(k_0) \frac{\eta_0}{k_0} \sigma \ \tilde{h}(T_{k_0},k_0,\eta). 
\end{align}
This contribution at frequency $k_0-1$ now in turn interacts with our underlying
solution $f_*= \epsilon \cos(x) \psi(v)$ and results in another \emph{echo} at
frequency $k_0-2$ during the time $(T_{k_0-1}, T_{k_0-1}')$. We thus obtain not
just a single echo but rather an \emph{echo chain}
\begin{align*}
  k_0 \rightarrow k_0-1 \rightarrow k_0-2 \rightarrow \dots \rightarrow 0
\end{align*}
along the sequence of times $T_{k_0}, T_{k_0-1}, \dots$ which results in
\begin{align}
  \tilde{h}(T_1',0,\eta) \approx \tilde{h}(T_{k_0},k_0,\eta) \prod_{j=1}^{k_0}\hat{W}(j) \frac{\eta_0}{j} \sigma. 
\end{align}
Approximating the last factor by Stirling's formula and choosing $k_0=\sqrt[3]{
  \sigma\eta_0}$ to maximize it, we obtain
\begin{align*}
  \prod_{j=2}^{k_0}\hat{W}(j) \frac{\eta_0}{j} \sigma \approx \exp(\sqrt[3]{\sigma \eta}),
\end{align*}
which suggests Gevrey 3 regularity as a critical class (see also
\cite{bedrossian2013landau}, \cite{bedrossian2016nonlinear} for a similar
derivation for a toy model of the nonlinear problem). We stress that this growth
is due to the mode $k=-1$ in $f_0$ repeatedly interacting with the echoes
created by $h$ and thus resulting in a self-sustained \emph{echo chain} of
maximal length.

In the following we will show that this toy model is accurate in the sense that
also our linear problem \eqref{eq:self} exhibits echo chains of maximal length.

\section{Echo Chains along Self-similar Solutions}
\label{sec:better}
Building on the intuition developed in the toy model of Section \ref{sec:echo}
we now return to the linearized Vlasov-Poisson equations around self-similar
solutions.

We recall that the full Vlasov-Poisson equations on $\T^d \times \R^d$ are given
by
\begin{align}
  \label{eq:VP}
  \dt f + (v , F) \cdot \nabla_{x,v} f =0.
\end{align}
By the results of Section \ref{sec:selfsimilar}, if $\psi \in \mathcal{S}(\R^d)$
has compact Fourier support inside $B_{\delta}(0)$, then the function
\begin{align}
  f_*(t,x,v)= f_0(v) + \epsilon \cos(e_1\cdot(x-tv)) \psi(v)
\end{align}
is a solution of the Vlasov-Poisson equations \eqref{eq:VP} on
$(\delta,\infty)$.

Considering a perturbation $f=f_* + \epsilon' h$ we obtain
\begin{align}
  \dt h + v \cdot \nabla_x h + F[h] \cdot \nabla_v f_* + \epsilon' F[h]\cdot \nabla_v h =0,
\end{align}
where we ignore the last part in the linearization. We next change coordinates
to $(x+tv,v)$ and in these new coordinates we obtain
\begin{align}
  \label{eq:LVP}
  \tag{LVP}
  \dt h + F[\int h(t,x-tw,w)dw] (\nabla_{v}-t \nabla_x)(f_0(v)+\epsilon \cos(x_1) \psi(v))=0.
\end{align}
Let now $\eta \in \R^d, k \in \Z^d$, then the Fourier transform satisfies
\begin{align}
  \label{eq:7}
  \begin{split}
    \dt \tilde{h}(t,k,\eta)&= - \hat{F}(t,k) \cdot (\eta-kt) \hat{f}_0(\eta-kt) \\
    & \quad - \sum_{l=k \pm 1} \hat{F}(t,l) (\eta-lt - t\text{sgn}(l-k)) \hat{\psi}(\eta-lt), \\
    \hat{F}(t,k)&= k \hat{W}(k) \tilde{h}(t,k,kt).
  \end{split}
\end{align}
We remark that $\hat{\psi}(\eta-lt)$ is supported in the set where
$|\eta-lt|\leq \delta \Leftrightarrow |\frac{\eta}{l}-t|< \frac{\delta}{|l|}$.
We will later use that for $\eta$ large and $t$ inside a small interval this can
only be satisfied for at most one $l$ and that for $t$ large it holds that $|\eta-lt| + t \approx t$.

If $\tilde{\psi}$ (or $\sigma$ in the toy model) is sufficiently small it seems
reasonable to expect that on a given time interval $(T_{k_0}, T_{k_0}')$ the first
Duhamel iteration is a good approximation of the evolution by the Vlasov-Poisson
equations. Indeed, this is true and shown in \cite{bedrossian2016nonlinear} for
the setting of plasma physics near homogeneous solutions and in \cite{dengZ2019} for the setting of fluid
dynamics. However, ``sufficiently small'' here turns out to be a restriction
depending on $\eta_0$ which roughly speaking is of the form
\begin{align*}
  \sigma < \frac{\log(\eta_0)}{\eta_0}.
\end{align*}
For example, condition (2.10) in \cite{bedrossian2016nonlinear} reads
\begin{align*}
  \sigma \eta_0^{-R} \exp(\sqrt[3]{\sigma \eta_0})=1.
\end{align*}
While this restriction allows to consider relatively large $\eta_0$ and derive
norm inflation, for given $\sigma$ we may only consider $\eta_0$ inside a
compact set and the possible growth is bounded in terms of $\sigma$. In
particular, with such a constraint one cannot consider sequences of $\eta_0$
tending to infinity which is crucial to construct critical spaces and solutions
which exhibit blow-up and convergence at the same time. Furthermore, following
the heuristic of the toy model of Section \ref{sec:echo} the time of the last
resonance $T_1 \approx \frac{\eta_0}{1}$ is uniformly bounded and thus of
limited use in predicting asymptotic behavior as $t\rightarrow \infty$.\\ 

Thus, as a main result of \cite{dengZ2019} Deng and the author showed that such a restriction can be
removed in the linearized fluids setting near Couette flow and that the dynamics are qualitatively
different without such a constraint.

In the present setting of the linearized Vlasov-Poisson equations the compact
support and simpler structure of the mapping $h \mapsto F$ allows for an easier
control of perturbations and a more transparent view of the resonance chain
mechanism. We remark that in \cite{bedrossian2016nonlinear} Bedrossian
considered $\tilde{\psi}$ with exponential decay but without compact support, which does not allow for this simplification.

In the following we consider the one-dimensional problem and with slight abuse
of notation write $h(t,k,\eta)$ instead of $\tilde{h}(t,k,\eta)$. For later
reference we also note the Duhamel integral formulation of \eqref{eq:7}:
\begin{align}
  \label{eq:Duhamel}
  \begin{split}
    \tilde{h}(T_1,k,\eta)- \tilde{h}(T_0,k,\eta) &= - \int_{T_0}^{T_1} k \hat{W}(k) \tilde{h}(t,k,kt) \cdot (\eta-kt) \hat{f}_0(\eta-kt) \\
    & \quad - \sum_{l=k \pm 1} l \hat{W}(l) \tilde{h}(t,l,lt) (\eta-lt -
    t\text{sgn}(l-k)) \hat{\psi}(\eta-lt) dt.
  \end{split}
\end{align}

The following Lemmas \ref{lem:obs1} and \ref{lem:obs3} use that the
right-hand-side of \eqref{eq:7} and \eqref{eq:Duhamel} only depends on
$h(t,\cdot,\cdot)$ in terms of some frequencies and due to the compact support
of $\hat{\psi}$ in turn only changes some frequencies.

\begin{lem}[Dependence on Initial Data]
  \label{lem:obs1}
  Let $I\subset \R$ be a given time interval and for any $l \in \N$ define
  \begin{align*}
    U_l:= \{lt: t \in I\} = l I.
  \end{align*}
  Let $\mathbb{P}_l$ denote the characteristic function of $U_l$. Then the
  evolution of $h$ by \eqref{eq:LVP} on the time interval $I$ depends only on
  the initial data restricted to $\{l\} \times U_l$. That is, if $h$ is the
  solution of equation \eqref{eq:LVP} on $I$ with initial data $h_0$ and $h_*$
  is the solution with initial data $h_{0,*}(l,\eta) = P_l(\eta)h_0(l,\eta)$,
  then
  \begin{align*}
    h(t,l,\eta)- h_*(t,l,\eta)= (1-P_l) h_0(l,\eta) 
  \end{align*}
  for all $t \in I$.
\end{lem}
\begin{proof}
  Let $h$ be a given solution of \eqref{eq:LVP}. Since $P_l$ and $h_0$ do not
  depend on time
  \begin{align*}
    \dt h = \dt (h- (1-P_l) h_0(l,\eta)).
  \end{align*}
  On the other hand the right-hand-side of equation \eqref{eq:7},
  \eqref{eq:Duhamel} depends on $h$ only in terms of $h(t,l,lt)$ and by
  construction $lt \in U_l$. Thus also,
  \begin{align*}
    \dt h = \rhs[h]= \rhs[h-(1-P_l)h_0].
  \end{align*}
  Therefore, we may define
  \begin{align*}
    h_*= h- (1-P_l)h_0
  \end{align*}
  and observe that $h_*$ is a solution of \eqref{eq:7} with initial data
  \begin{align*}
    h_{0,*}= h_0- (1-P_l)h_0= P_l h_0.
  \end{align*}
\end{proof}

Lemma \ref{lem:obs1} in a sense allows us to exchange compact support in
frequency for compact support in time. For example in Proposition
\ref{prop:init} we consider initial data which is supported on $\{k_0\} \times
(\eta_0-1/2,\eta_0+1/2)$. Thus, if we define $I= \{t: k_0 t \leq \eta_0-1/2\}$,
then the set
$\{k\} \times U_{k_0}$ and the support of our initial data are disjoint and $\dt
h_*$ vanishes. Furthermore, as we show in Proposition \ref{prop:iterate} this
Lemma allows us to ``forget'' about all but the most recent echo in the chain.

We remark that Lemma \ref{lem:obs1} did not yet use any properties of $f_0$ or
$\psi$. In the following we restrict to the special case $f_0\equiv 0$ and
$\psi$ compactly supported in Fourier space inside a ball $B_{\delta}(0)$.

\begin{lem}[Domains of Dependence]
  \label{lem:obs3}
  Let $I \subset \R$ be a given interval and define
  \begin{align*}
    U_l = l I , U_{l}^{\delta}= U_l + (-\delta,\delta).
  \end{align*}
  Then for any $l \in \N$ and any $t \in I$, the solution $h$ of equation
  \eqref{eq:7} with $f_0\equiv 0$ satisfies
  \begin{align*}
    \text{supp}(\dt h(t,l,\cdot)) \subset U_{l-1}^\delta \cup U_{l+1}^\delta. 
  \end{align*}
  In particular, if for some $l$ it holds that
  \begin{align*}
    U_l \cap (U_{l-1}^\delta \cup U_{l+1}^\delta)= \emptyset,
  \end{align*}
  then
  \begin{align*}
    h(t,l,lt)= h(t_0,l,lt)
  \end{align*}
  for any $t_0 \in \overline{I}$ (e.g. the left endpoint).
\end{lem}
\begin{proof}
  We note that $lt \in U_{l}$ and that $\hat{\psi}(\eta-lt)$ vanishes unless $|\eta-lt|<\delta$.
\end{proof}

Building on these two Lemmas we are now ready to construct our chain of
echoes.
\begin{prop}[Initial Setup]
  \label{prop:init}
  Let $\eta_0 \gg k_0\geq 1$ (i.e. greater by a factor at least $100$) be given
  and suppose that $h_0$ is supported in $\{k_0\} \times
  (\eta_0-1/2,\eta_0+1/2)$. Define $T_{k_0}>0$ by $\eta_0-1/2-k_0T_{k_0}=0$,
  i.e. the first time $(\eta_0-1/2, \eta_0+1/2)-k_0 t$ hits zero. Then on $(0,T_{k_0})$ the
  solution of \eqref{eq:LVP} with initial data $h_0$ is stationary.
\end{prop}
\begin{proof}[Proof of Proposition \ref{prop:init}]
  Let $I=(0,T_{k_0})$, then by Lemma \ref{lem:obs1} we may equivalently compute the
  solution $h_*$ with initial data $h_{*,0}=P_l h_0$. But by construction of $T_{k_0}$
  it holds that $P_k h_0(k,\cdot)=0$ and thus $h_{*,0}$ and hence $h_*$ are
  trivial and
  \begin{align*}
    h= h_* + (1-P_l)h_0=0 + h_0 
  \end{align*}
  is a stationary solution.
\end{proof}

\begin{prop}[The First Echo]
  \label{prop:firstres}
  Let $\eta_0,k_0$ and $h_0$ be as in Proposition \ref{prop:init}, $f_0\equiv 0$
  and define
  \begin{align*}
    U_{k_0}= (\eta_0-1/2, \eta_0+1/2), \\
    I_{k_0}= \frac{1}{k_0} U_{k_0}.
  \end{align*}
  Note that the left endpoint of $I_{k_0}$ is given by $T_{k_0}$ of Proposition
  \ref{prop:init}. Then it holds that for all $t \in I_{k_0}$
  \begin{align*}
    h(t,l,\eta)&\equiv 0 \text{ for all } l \not \in \{k_0-1,k_0,k_0+1\}, \\
    h(t,k_0,\eta)&= h(T,k_0,\eta), \\
    h(t,k_0\pm 1, \eta) &= \int_{T}^t  k_0 \hat{W}(k_0) h(T,k_0,k_0t) ((\eta-k_0t)\mp t) \hat{\psi}(\eta-k_0t) dt. 
  \end{align*}
\end{prop}

\begin{proof}[Proof of Proposition \ref{prop:firstres}]
  Since $\eta_0\gg k_0\geq 1$ it follows that for $l \in \{k_0-1,k_0+1\}$ in the
  notation of Lemma \ref{lem:obs3}
  \begin{align*}
    U_{l} \cap (U_{l+1}^{\delta} \cup U_{l-1}^\delta)=\emptyset.
  \end{align*}
  Thus, in the evolution of $h(t,l\cdot)$ for $l< k_0-1$ and $l>k_0-1$ we may
  replace $h(t,k_0\pm 1, (k_0\pm 1)t)= h(T, k_0 \pm 1,(k_0\pm 1)t)=0$, which is
  trivial. The evolution of $l< k_0-1$ and $l>k_0-1$ thus decouples from the
  evolution of $\{k_0-1,k_0,k_0+1\}$. In particular, since it was trivial to
  begin with it remains trivial, which proves the first statement.

  Furthermore, since $h(t,k_0\pm 1, (k_0\pm 1)t)=0$, the second statement follows
  immediately form the Duhamel integral formulation \eqref{eq:Duhamel}.
  
  Finally, we may either use Lemma \ref{lem:obs3} to obtain that
  \begin{align*}
    P_{k_0} h(t,k_0,\eta)= P_{l}h(T,k_0,\eta) 
  \end{align*}
  and thus in particular
  \begin{align*}
    h(t,k_0,k_0t)= h(T,k_0,k_0t),
  \end{align*}
  or deduce this from the second statement.
  The third statement thus also corresponds to an evaluation of the Duhamel
  integral formula \eqref{eq:Duhamel}.
\end{proof}

\begin{prop}[Pause between Echoes]
  \label{prop:stationaryuntil}
  Let $\eta_0, k_0, h_0, f_0$ be as in Proposition \ref{prop:init}. Let further
  $0<T_{k_0}'<T_{k_0-1}$ such that:
  \begin{align*}
    \eta_0+1/2-k_0T_{k_0}'=0, \\
    \eta_0-1/2-\delta-(k_0-1)T_{k_0-1}=0.
  \end{align*}
  Then for all $t \in [T_{k_0}',T_{k_0-1}]$ it holds that
  \begin{align*}
    h(t)=h(T_{k_0}')
  \end{align*}
  is stationary.
\end{prop}

\begin{proof}[Proof of Proposition \ref{prop:stationaryuntil}]
  By the result of Proposition \ref{prop:init} at time $T=T_{k_0}'$ it holds that
  \begin{align}
    \label{eq:8}
    \begin{split}
      \supp(h(T,k_0,\cdot)) &\subset (\eta_0-1/2,\eta_0+1/2), \\
      \supp(h(T,k_0+1,\cdot))&\subset (\eta_0-1/2-\delta,\eta_0+1/2+\delta),\\
      \supp(h(T,k_0-1,\cdot))&\subset (\eta_0-1/2-\delta,\eta_0+1/2+\delta),
    \end{split}
  \end{align}
  and all other modes $h(T,l,\cdot)$ are trivial. We now apply Lemma
  \ref{lem:obs1} with $I=(T_{k_0}',T_{k_0-1})$ and observe that
  \begin{align}
    \label{eq:9}
    \begin{split}
      \eta_0 +1/2 &\leq k_0 T_{k_0}' , \\
      \eta_0 +1/2 + \delta &\leq (k_0+1) T_{k_0}', \\
      (k_0-1)T_{k_0-1} &\leq \eta_0-1/2-\delta.
    \end{split}
  \end{align}
  Hence all restrictions to $U_{k_0-1}, U_{k_0}$ and $U_{k_0+1}$ are trivial and
  we obtain a stationary solution.
\end{proof}

We remark that \eqref{eq:8} remains valid until the time $T_{k_0-1}$ and that the
estimates for $k_0$ and $k_0+1$ remain valid if we replace $T_{k_0}'$ with any larger
time. Hence, we may use Lemma \ref{lem:obs1} in combination with Proposition
\ref{prop:init} to study the next echo.

\begin{prop}[Iterating along a Chain]
  \label{prop:iterate}
  For $k=1,\dots k_0$ define
  \begin{align}
    \begin{split}
      T_k&: \eta_0-1/2-\delta (k_0-k) - k T_k=0, \\
      T_k'&: \eta_0+1/2+\delta (k_0-k) - k T_k'=0.
    \end{split}
  \end{align}
  Then on each interval $(T_k', T_{k-1})$, $k>1$ the solution $h(t)$ is
  stationary. For the evolution on $(T_k, T_k')$, $k\geq 1$, we note that for
  all $t \in (T_k,T_k')$:
  \begin{align}
    \label{eq:11}
    h(t, l, \cdot) &= h (T_{k},l,\cdot) \text{ if } l\neq k-1,k+1, \\
    h(t,k\pm 1, \eta)- h(T_k,k\pm 1, \eta)&= \int_{T_k}^t k \hat{W}(k) h(T,k,k\tau) ((\eta-k\tau)\mp \tau) \psi(\eta-k\tau) d\tau.
  \end{align}
  Furthermore, at time $T_{k'}$ for all $l=k_0+1,k_0,\dots, k-1$
  \begin{align}
    \label{eq:10}
    \supp(h(T_{k}',l,\cdot))&\subset (\eta_0-1/2-\delta |k_0-l|, \eta_0+1/2+\delta |k_0-l|).
  \end{align}
  and for all other $l$, $h(T_{k'},l,\cdot)$ is trivial.
\end{prop}

The time intervals $(T_k,T_{k}')$ are when we see \emph{echoes} and
\eqref{eq:11} expresses the corresponding growth in terms of Duhamel integral.
In Theorem \ref{thm:inflation} we show that a full chain may result in Gevrey norm
inflation.
However, we also note that after the time $T_1'$ this single echo chain solution becomes stationary (see
Proposition \ref{prop:asymptoticallystationary}) and is thus
\emph{asymptotically stable}. In order to construct asymptotically unstable
initial data and blow-up we thus need to combine countably infinitely many echo
chains with times $T_{1,j}'\rightarrow \infty$ (see Theorem \ref{thm:blowup}).
In particular, we cannot allow any constraint on the size of $\eta_0$.

\begin{proof}[Proof of Proposition \ref{prop:iterate}]
  We remark that \eqref{eq:10} and the triviality of all other modes
implies that $h$ is stationary on $(T_{k}', T_{k-1})$ by the same argument as in Proposition \ref{prop:stationaryuntil}.

  It thus remains to study the evolution on $I_k=(T_{k},T_{k}')$. We proceed by
  induction in $k$ and have already established the case $k=k_0$ in Proposition
  \ref{prop:firstres}. Thus suppose that \eqref{eq:10} and \eqref{eq:11} hold
  for a given $k>1$ (if $k=1$ we are already done). Then by the above argument
  the evolution on $(T_{k}', T_{k-1})$ is trivial and hence \eqref{eq:10}
  remains valid with $T_{k'}$ replaced by $T_{k-1}$. Using Lemma \ref{lem:obs1}
  with $I=(T_{k-1}, T_{k-1}')$ we may further reduce the problem to the one with
  $h(T_{k-1})$ replaced by its restriction. Since $\eta_0\gg k_0\geq k$ it
  follows that for all $l> k-1$
  \begin{align*}
    \eta_0+1/2+\delta |k_0-l| \leq \frac{l}{k-1} \left(\eta_0-1/2-\delta (k_0-k)  \right)\\
    =  \frac{l}{k-1} (k-1) T_{k-1}= l T_{k-1}. 
  \end{align*}
  Therefore, for these $l$ the restrictions are trivial and we may solve for $h_*$
  with initial data $h_*(T_{k-1}, l,\cdot)= \delta_{l (k-1)} h_(T_k, l,\cdot)$.
  The initial data of $h_*$ is localized in a \emph{single} mode $k-1$.
  We may thus repeat the argument of Proposition \ref{prop:firstres} to show that  $h_*(t,l,\cdot)$ is trivial unless $l= (k-1)\pm 1$ and for those $l$ it is
  given by the Duhamel integral, which is the statement of \eqref{eq:11} for
  $k-1$. Finally, we note that the support for $l \neq (k-1)\pm 1$ is preserved
  and that due to the convolution structure in terms of $\tilde{\psi}(\eta-(k-1)t)$
  the support of the Duhamel integral is contained inside a $\delta$
  neighborhood of the support of $h(T_{k-1},k-1,\cdot)= h(T_k', k-1,\cdot)$,
  which was controlled by the induction assumption.
\end{proof}

\begin{prop}[Asymptotic stability]
  \label{prop:asymptoticallystationary}
  Let $h_0$ and $T_{j}'$ be as in Proposition \ref{prop:iterate}. Then the
  solution $h(t)$ is stationary after the time $T_1'$. In particular, denoting
  $h_{\infty}=h(T_1')$ it trivially holds that
  \begin{align*}
    h(t) \rightarrow  h_{\infty}
  \end{align*}
  as $t \rightarrow \infty$ in any Sobolev or Gevrey space which contains $h(T_1')$.
\end{prop}

\begin{proof}[Proof of Proposition \ref{prop:asymptoticallystationary}]
  We apply Lemma \ref{lem:obs1} with $I=(T_1',\infty)$. Then the projection of
  $h(T_{1}',l,\eta)$ is trivial for $l\neq 0$. Since $0 \hat{W}(0)=0$ the mode
  $h(t,0,\cdot)$ does not influence any mode (not even itself).
  Hence, $h_*$ is a stationary solution and thus so is $h$.
\end{proof}

\section{Sequences of Chains, Norm Inflation and Blow-up}
\label{sec:blow-up}

In the preceding Proposition \ref{prop:iterate} we have seen that a solution
with initial Fourier support near $(k_0,\eta_0)$ with $\eta_0\gg k_0\geq 1$ stays
supported in a stripe
\begin{align*}
  \{1,\dots, k_0+1\} \times (\eta_0-1/2-k_0\delta, \eta_0+1/2 +k_0\delta).
\end{align*}
For such a function all Sobolev norms or Gevrey norms are equivalent to the
$L^2$ norm with a corresponding factor $\eta_0^s$ or $\exp(C \eta_0^{1/s})$,
respectively. Thus the following $L^2$ norm inflation result immediately extends
to other ($L^2$-based) Sobolev or Gevrey norms.

\begin{thm}
  \label{thm:inflation}
  Let $\psi(v) \in \mathcal{S}(\R)$ with compact support in Fourier space,
  $\hat{\psi}\geq 0$ and $|\hat{W}(k)|=|k|^{1-s}, s>0$.

  Then there exist pairs $(\eta_0,k_0)$ tending to infinity and initial data $h_0 \in L^2$
  supported in $\{k_0\} \times (\eta_0-1/2,\eta_0+1/2)$ such that the solution
  $h(t)$ with this initial data is stationary for $t>\eta_0+1/2+ \delta
  k_0=:T_1'$ and there exist constants $c_1,c_2$ (proportional to
  $\|\psi(v)\|_{L^\infty}$ and independent of $\eta_0$) such that
  \begin{align*}
 \exp(\sqrt[s]{c_1 \eta_0})\leq   \|h(T_1')\|_{L^2} \leq \exp(\sqrt[s]{c_2 \eta_0})
  \end{align*}
\end{thm}
In particular, we see that the linearized Vlasov-Poisson equations around
self-similar solutions \eqref{eq:LVP} exhibit \emph{full echo chains} of
\emph{arbitrarily long length}.
They thus exhibit norm inflation in any Sobolev or
Gevrey norm by an arbitrarily large factor (which is not limited by a smallness
constraint!).
In Theorem \ref{thm:blowup} we further show that there exists Gevrey regular
initial data which not only exhibits norm inflation but \emph{blow-up}.
As a complementary result in Theorem \ref{thm:stability} we show that Gevrey 3
regularity is critical in the sense that this blow-up corresponds to a loss of
constant in the exponent and can be absorbed in the case of highly regular data
(that is, a large constant in the exponent).
The problem \eqref{eq:LVP} is \emph{stable} in (high) Gevrey 3 regularity.

\begin{proof}[Proof of Theorem \ref{thm:inflation}]
  The solution constructed in the preceding section is stationary after time
  $T_1'$ and satisfies the support assumptions. It thus only remains to choose
  $\eta_0,k_0$ and $h_0$ in a suitable way to obtain upper and lower bounds.

  Using the identities obtained in Proposition \ref{prop:iterate} $k_0$ times we
  note that
  \begin{align*}
    h(T_1', 0, \eta)&= \int_{T_1}^{T_1'} 1 \hat{W(1)} (\eta-\tau_1 + \tau_1) \psi(\eta-\tau_1) h(T_1,1,\tau_1) d\tau_1 \\
                    &= \int_{T_1}^{T_1'} 1 \hat{W}(1) (\eta-\tau_1 + \tau_1) \psi(\eta-\tau_1)\\
                    & \quad \int_{T_2}^{T_2'} 2 \hat{W}(2) (\tau_1-2\tau_2 + \tau_2) \psi(\tau_1-2 \tau_2) h(T_2, 2, 2 \tau_2) d\tau_1 d \tau_2 \\
                    &= \left( \prod_{j=1}^{k_0} j \hat{W}(j)\right) \\
                    & \quad \int_{T_{j}\leq \tau_j \leq T_{j}'}  (\eta-\tau_1 + \tau_1)(\tau_1-2\tau_2 + \tau_2) \dots \\
                    & \quad \psi(\eta-\tau_1) \psi(\tau_1-2 \tau_2) \dots  \\
                    & \quad h_0(0, k_0, k_0 \tau_{k_0}) d\tau_1 \dots d\tau_{k}.
  \end{align*}

  To simplify estimates we assume that $h_0(0,k_0,\cdot)\geq 0$ and $\psi\geq
  0$. By the support assumption on $\psi$ it further holds that $|\eta-\tau_1|<
  \delta, |\tau_1-2\tau_2|< \delta, \dots$ and thus
  \begin{align*}
    0\leq T_{k+1}-\delta \leq k \tau_k - (k+1)\tau_{k+1} + \tau_{k+1} \leq T_{k}' +\delta.
  \end{align*}
  Therefore we may bound $h(T_1', 0, \eta)$ above and below in terms of
  \begin{align}
    \label{eq:13}
    \left( \prod_{j=1}^{k_0} j \hat{W}(j) (T_j-\delta)  \right)
  \end{align}
  and
  \begin{align}
    \label{eq:14}
    \left( \prod_{j=1}^{k_0} j \hat{W}(j) (T_j'+\delta)  \right)
  \end{align}
  times
  \begin{align}
    \int_{T_{j}\leq \tau_j \leq T_{j}'} \psi(\eta-\tau_1) \psi(\tau_1-2 \tau_2) \dots \psi((k_0-1)\tau_{k_0-1}- k_0\tau_{k_0}) \ h_0( k_0, k_0 \tau_{k_0}) d\tau_1 \dots d\tau_{k}.
  \end{align}
  We next introduce a change of variables $s_j=j \tau_j$, which yields a Jacobian
  determinant of $\frac{1}{k_0!}$ to obtain an iterated convolution
  \begin{align}
    \label{eq:12}
    \frac{1}{k_0!} \int_{jT_{j}\leq s_j \leq j T_{j}'} \psi(\eta-s_1) \psi(s_1-s_s) \dots  h_0( k_0, s_{k_0}) ds.
  \end{align}
  By the support assumption on $h_0$ we may further replace the domain of
  integration by all of $\R^{k_0}$. Therefore, the $L^2$ norm of \eqref{eq:12}
  can be computed by Plancherel's theorem as
  \begin{align*}
    \frac{1}{k_0!} \|\check{\psi}(v)^{k_0} h_0(k_0,v)\|_{L^2}\approx \frac{\|\check{\psi}\|_{L^\infty}^{k_0}}{k_0!} \|h_0\|_{L^2} .
  \end{align*}
  We further note that by the construction of $T_j$
  \begin{align*}
    T_j-\delta \approx \frac{\eta_0}{j} \approx T_j'+ \delta. 
  \end{align*}
  Using this approximation in \eqref{eq:13} and \eqref{eq:14} with constants
  $C_1, C_2$ we may thus estimate
  \begin{align*}
   C_2^{k_0}\|\check{\psi}\|_{L^\infty}^{k_0} \prod_{j=1}^{k_0} \hat{W}(j) \frac{\eta_0}{j} \leq  \|h(T_1', 0 ,\eta)\|_{L^2} \leq C_2^{k_0}\|\check{\psi}\|_{L^\infty}^{k_0} \prod_{j=1}^{k_0} \hat{W}(j) \frac{\eta_0}{j},
  \end{align*}
  where we cancelled the product over $j$ and the $\frac{1}{k_0!}$.
  As in \cite{Villani_long}, \cite{bedrossian2013landau} we note that if
  $\hat{W}(j)=j^{1-s}$, then
  \begin{align*}
    C_2^{k_0}\|\check{\psi}\|_{L^\infty}^{k_0} \prod_{j=1}^{k_0} \hat{W}(j) \frac{\eta_0}{j} = (C_2 \|\check{\psi}\|_{L^\infty} \eta_0)^{k_0} (k_0!)^{-s}
  \end{align*}
  can be approximated by Stirling's formula and attains its maximum
  \begin{align*}
    \exp(\sqrt[s]{C_2 \|\check{\psi}\|_{L^\infty} \eta_0})
  \end{align*}
  for $k_0\approx \sqrt[s]{C_2 \|\check{\psi}\|_{L^\infty} \eta_0}$.
\end{proof}

\begin{thm}[Blow-up]
  \label{thm:blowup}
  Let $|\hat{W}(k)|=|k|^{-2}$.
  For every $s \in \R$ there exists $h_\infty \in H^{s}\setminus H^{s+}$ and
  $h_0 \in \mathcal{G}_{3}$ such that the solution $h(t)$ of \eqref{eq:LVP} with
  initial datum $h_0$ converges to $h_\infty$ in $H^s$ as $t \rightarrow
  \infty$. In particular $h(t)$ \emph{diverges} in $H^{s+}$. However, if $s\geq
  0$ then $F[h](t)\rightarrow_{L^2} 0$ as $t\rightarrow \infty$ and thus
  \emph{damping persists}.
\end{thm}

\begin{proof}[Proof of Theorem \ref{thm:blowup}]
  Let $\eta_{0,j}, k_{0,j}$, $h_{0,j}$ be a sequence of frequencies and initial
  data as in Theorem \ref{thm:inflation} with $\|h_{0,j}\|_{L^2}=1$ and let
  $h_j(t)$ be the corresponding solutions and $h_{j,\infty}$ the asymptotic
  state. After possibly choosing a subsequence we may further assume that the
  sets
  \begin{align*}
    (\eta_{0,j}-1/2-\delta k_{0,j},\eta_{0,j}+1/2+\delta k_{0,j})
  \end{align*}
  are all disjoint and hence by Proposition \ref{prop:iterate} all $h_{j}(t)$
  are $L^2$ orthogonal for all times $0<t\leq \infty$. Furthermore, by Theorem
  \ref{thm:inflation} there exists $C_j \approx \exp(\sqrt[3]{c
    \eta_{0,j}})$ such that
  \begin{align*}
    \lim_{t\rightarrow \infty} \|h_{j}(t)\|_{L^2} = \|h_{j,\infty}\|_{L^2}= C_j \|h_{0,j}\|_{L^2}=C_j,
  \end{align*}
  and there exist times $T_j$ after which $h_j(t)$ is stationary.
  
  Let next $(\alpha_j)_j \in l^2$ but such that for no $\sigma>0$,
  $\eta_{0,j}^{\sigma} \alpha_j \in l^2$. For a given $s \in \R$ the solution
  with initial data
  \begin{align}
    h_0= \sum \frac{1}{C_j} \eta_{0,j}^{-s} \alpha_{j} h_{0,j}
  \end{align}
  is given by
  \begin{align*}
    h(t)= \sum \frac{1}{C_j} \eta_{0,j}^{-s} \alpha_{j} h_{j}(t).
  \end{align*}
  Since the functions $h_{0,j}$ are disjointly supported in Fourier space, $L^2$
  normalized, and concentrated near $\eta_{0,j}$, we obtain that for any $c'<1$
  \begin{align*}
    \int \exp(2 \sqrt[3]{c' |\eta|}) |\tilde{h}_0|^2 \leq C \sum_{j}\exp(2c\sqrt[3]{ c'|\eta_{0,j}|} ) \frac{1}{C_j^2} \eta_{0,j}^{-2s} |\alpha_{j}|^2 <\infty,
  \end{align*}
  where we used that growth of $C_j$ dominates for $j\rightarrow \infty$ and
  that $\alpha \in l^2$. In order to show the asymptotic convergence we note
  that after the time $T_J$
  \begin{align*}
    h(t)= \sum_{j\leq J} \eta_{0,j}^{-s} \alpha_{j} \frac{h_{j,\infty}}{C_j} \\
    + \sum_{j\geq J} \eta_{0,j}^{-s} \alpha_{j} \frac{h_{j}(t)}{C_j}.
  \end{align*}
  The first sum is stationary for all future times and $\eta_{0,j}^{-s}
  \frac{h_{j,\infty}}{C_j}$ is (approximately) $H^s$ normalized. For the
  remaining integral we note that $\eta_{0,j}^{-s}\frac{h_{j}(t)}{C_j}$ is
  uniformly bounded in $H^{s}$ and disjointly supported in Fourier space and
  thus
  \begin{align*}
    \|\sum_{j\geq J} \eta_{0,j}^{-s} \alpha_{j} \frac{h_{j}(t)}{C_j}\|_{H^s}^2 \leq C \sum_{j\geq J} |\alpha_j|^2 \rightarrow 0 
  \end{align*}
  as $J \rightarrow \infty$.
\end{proof}
We have thus constructed a solution with Gevrey $3$ regular initial data which
not only exhibits norm inflation but \emph{blow-up} in any Sobolev regularity.
Therefore ``strong Landau damping'' in the sense \ref{item:strong}, that is
scattering to free transport, fails.
However, ``physical Landau damping'' in the sense \ref{item:physical}, that is the convergence of the force field,
\emph{persists}!
Here, it was crucial to be able to consider not just a single \emph{echo chain}
as in Theorem \ref{thm:inflation} but infinitely many.
That is, any single chain not matter how long is asymptotically stable after a
finite time by Proposition \ref{prop:asymptoticallystationary}. Hence, in order
to obtain non-trivial asymptotic behavior like blow-up we need to construct
\emph{sequences of echo chains} which become resonant at later and later times
and which become longer and longer.
In particular, we may not require any constraint on the size of $\eta_0$ which
remains the main obstacle in extending the preceding results to other linearly
stable homogeneous states $f_0(v)$ or the nonlinear dynamics.

\section{Stability in Gevrey 3 with Large Constant}
\label{sec:stability}

As a complementary result to the instability constructions of Section
\ref{sec:blow-up} in this section we show that if
\begin{align*}
  \sum_{k}\int |\tilde{h_0}(k,\eta)|^2 \exp(C \sqrt[3]{|\eta|}) d\eta<\infty
\end{align*}
for a suitable constant $C>1$, then $h(t) \in \mathcal{G}_3$ for all times and $h_{\infty} \in
\mathcal{G}_3$. Thus the Gevrey $3$ class is \emph{critical} for the linear
problem \eqref{eq:LVP} just as it is in the nonlinear problem.

We remark that the upper bounds of Theorem \ref{thm:inflation} are also valid
for generic frequency-localized initial data (without positivity assumption).
\begin{lem}[Upper Bound]
  \label{lem:upper}
  Let $\psi \in \mathcal{S}(\R)$ with Fourier support in $(-\delta,\delta)$ and
  $\hat{\psi}\geq 0$. Let further $h_0 \in L^2$ supported in $\{k_0\} \times
  (\eta_0-1/2,\eta_0+1/2)$.
  Then the solution $h(t)$ of \eqref{eq:LVP} is stationary for $t> \eta_0+1/2+
  \delta k_0=:T_1'$, supported in $\{0,\dots, k_0+1\} \times
  (\eta_0-1/2-k_0\delta, \eta_0+1/2+k_0 \delta)$ and for a constant $c_2>0$ it
  holds that
  \begin{align*}
    \|h(T_1')\|_{L^2} \leq \exp(\sqrt[3]{c_2 \eta_0})\|h_0\|_{L^2}. 
  \end{align*}
\end{lem}

\begin{proof}[Proof of Lemma \ref{lem:upper}]
 We recall from the proof of Theorem \ref{thm:inflation} that $h(T_1',0,\eta)$
 can be explicitly computed in terms of the initial data:
\begin{align*}
  h(T_1',0,\eta) &= \left( \prod_{j=1}^{k_0} j \hat{W}(j)\right) \\
                    & \quad \int_{T_{j}\leq \tau_j \leq T_{j}'}  (\eta-\tau_1 + \tau_1)(\tau_1-2\tau_2 + \tau_2) \dots \\
                    & \quad \psi(\eta-\tau_1) \psi(\tau_1-2 \tau_2) \dots  \\
                    & \quad h_0(0, k_0, k_0 \tau_{k_0}) d\tau_1 \dots d\tau_{k}.
\end{align*}
Similarly also $h(T_1',k,\eta)$ for $k=1,\dots, k_0+1$ can be explicitly
computed in terms of a $|k_0-k|$-fold integral. Since those estimates are
analogous (and smaller) we focus on the case $k=0$.
Here, we may estimate from above by absolute values and using \eqref{eq:14} to obtain that
\begin{align*}
  |h(T_1',0,\eta)|&\leq \left( \prod_{j=1}^{k_0} j |\hat{W}(j)| (T_j'+\delta)\right) \\
                  & \quad  \int_{T_{j}\leq \tau_j \leq T_{j}'} |\psi|(\eta-\tau_1) |\psi|(\tau_1-2 \tau_2) \dots |\psi|((k_0-1)\tau_{k_0-1}- k_0\tau_{k_0}) \\
  & \quad |h_0|( k_0, k_0 \tau_{k_0}) d\tau_1 \dots d\tau_{k}.
\end{align*}
Again using the support assumption on $h_0$ to replace the domain of integration
by all of $\R^{k_0}$ and changing variables to $s_j=j \tau_j$, we obtain 
\begin{align}
  \label{eq:2}
  \frac{1}{k_0!} \left( \prod_{j=1}^{k_0} j \hat{W}(j)(T_j'+\delta) \right)
\end{align}
times the $k_0$-fold convolution
\begin{align}
  \label{eq:4}
  \int |\psi|(\eta-s_1) |\psi|(s_1-s_s) \dots  |h_0|( k_0, s_{k_0}) ds.
\end{align}
By Plancherel's theorem the $L^2$ norm of this integral (which due to the
compact support in frequency is comparable to any Sobolev or Gevrey norm) is
equal to
\begin{align*}
  \| (\mathcal{F}^{-1}(|\psi|))^{k_0} \mathcal{F}^{-1}(|h_0|)\|_{L^2} \leq \|\mathcal{F}^{-1}(|\psi|)\|_{L^\infty}^{k_0} \| \mathcal{F}^{-1}(|h_0|)\|_{L^2}\|= \|\mathcal{F}^{-1}(|\psi|)\|_{L^\infty}^{k_0} \| h_0\|_{L^2}\|.
\end{align*}
Denoting $\|\mathcal{F}^{-1}(|\psi|)\|_{L^\infty}$ as $c_2$ and estimating
\eqref{eq:2} as in Theorem \ref{thm:inflation} we thus obtain the desired upper
bound by using Stirling's formula.
\end{proof}

In order to pass from a stability result for frequency-localized data to one for
general data we use the linearity of the equation and the control of the support
of solutions.
Since we required that $\eta_0\gg k_0$, here it is advantageous to first
consider the problem starting at time $t=100$.
\begin{thm}
  \label{thm:stability}
  Let $h_{100} \in \mathcal{G}_{3}$ be such that
  \begin{align}
    \label{eq:17}
  \sum_{k}\int |\tilde{h}_{100}(k,\eta)|^2 \exp(2 \sqrt[3]{|\eta|}) d\eta=C_{100}<\infty.
  \end{align}
  Then there exists a constant $C$ (independent of $h_{100}$ or $C_{100}$) such
  that for all times $t>100$
  \begin{align}
    \label{eq:18}
    \sum_{k}\int |\tilde{h}(t,k,\eta)|^2 \exp(1 \sqrt[3]{|\eta|}) d\eta \leq C C_{100}.
  \end{align}
\end{thm}

\begin{proof}[Proof of Theorem \ref{thm:stability}]
  Let $h_{100} \in \mathcal{G}_{3}$ be given. By the arguments of Section
  \ref{sec:better} the evolution of $k>0$, $k=0$ and $k<0$ decouple and the
  evolution for $\eta$ and $k$ having opposite signs is trivial.
  In the following we may thus without loss of generality assume that $k>0,\eta\geq 0$.
  We then define the set
  \begin{align*}
    \Omega_{100}=\{(k,\eta): k \geq 100 \eta\}.
  \end{align*}
  By Lemma \ref{lem:obs1} the solution with initial data $1_{\Omega_{100}}
  h_{100}$ is stationary.
  We thus focus on the evolution of the initial data in the complement.
  Here, we may further partition with respect to $\eta$ and $k$:
  \begin{align}
    \label{eq:5}
    \begin{split}
    \Omega^j &= \{(k,\eta): k \leq 100 \eta, \eta \in (j-1/2,j+1/2)\},\\ 
    \Omega^{j,k} &= \{(k,\eta): \eta \in (j-1/2,j+1/2)\} \text{ for } k\leq 100 \eta. 
    \end{split}
  \end{align}
  If $f_j(t)$ denotes the solution with initial data $1_{\Omega_j}h_{100}$ then
  by Proposition \ref{prop:iterate} $f_j(t)$ is supported in $(j-1/2-\delta
  \frac{j}{100}, j+1/2+ \frac{j}{100})$ and thus has overlap with only about $j$
  of its neighbors.
  Therefore, we may estimate
  \begin{align}
    \label{eq:15}
    \|f(t)\|_{L^2}^2 = \|1_{\Omega_{100}}
  h_{100} + \sum_j f_j(t)\|_{L^2}^2 \leq 2\|1_{\Omega_{100}}h_{100}\|_{L^2}^2 + \sum_{j} j^2\|f_j(t)\|_{L^2}^2 
  \end{align}
  and due to the frequency localization analogous estimates also hold with
  Sobolev or Gevrey norms in place of $L^2$.
  It thus remains to estimate $\|f_j(t)\|_{L^2}$ in terms of the initial data.
  Here, we may further split according to \eqref{eq:5} into
  \begin{align}
    \label{eq:16}
    \begin{split}
    f_j(t)&= \sum_{1\leq k\leq \frac{j}{100}} f_{j,k}(t) \\
    \leadsto \|f_j(t)\|_{L^2}^2 &\leq j \sum_{k} \|f_{j,k}(t)\|_{L^2}^2.
    \end{split}
  \end{align}
  Since each $f_{j,k}$ is highly frequency-localized we may apply Lemma
  \ref{lem:upper} to estimate
  \begin{align*}
    \|f_{j,k}(t)\|_{L^2} \leq \exp(\sqrt[3]{c_2 |j|}) \|1_{\Omega_{j,k}} f_{100}\|_{L^2}.
  \end{align*}
  The loss of powers of $j$ in \eqref{eq:15} and \eqref{eq:16} and the loss of
  the factor $\exp(\sqrt[3]{c_2 |j|})$ can easily be absorbed into the loss
  of $\exp(1 \sqrt[3]{ |j|})$ from \eqref{eq:17} to \eqref{eq:18}. Thus,
  indeed $f(t) \in \mathcal{G}_{3}$ with contant $1$ uniformly in time and the
  convergence to $h_{\infty}$ follows from the convergence of the partial sums
  in $j$ and $k$ (see also the proof of Theorem \ref{thm:blowup}).
\end{proof}

It remains to estimate the growth for the finite time interval $(0,100)$, where
we only need a rough upper estimate.
\begin{lem}
  \label{lem:local}
  There exists a constant $C$ such that if $h_0 \in \mathcal{G}_{3}$ with
  \begin{align*}
  \sum \int |\tilde{h}_0|^2 \exp(2 \sqrt[3]{|\eta|}) d\eta\leq C_{0}<\infty,
  \end{align*}
  then
  \begin{align*}
    \sum \int |\tilde{h}(100)|^2 \exp(2 \sqrt[3]{|\eta|}) d\eta\leq C C_{0}.
  \end{align*}
\end{lem}

\begin{proof}[Proof of Lemma \ref{lem:local}] 
We remark that by Lemma \ref{lem:obs1} we only need to consider initial data with
\begin{align*}
  \eta \leq 100 k,
\end{align*}
and that by the Fourier formulation of \eqref{eq:LVP} and the compact
support of $\psi$ we see that $h(100)-h_0$ is supported in
\begin{align*}
  \eta \leq 100 k + \delta.
\end{align*}

We recall that the Duhamel integral formulation of \eqref{eq:LVP} is given by \eqref{eq:Duhamel}
\begin{align*}
    h(T_1,k,\eta)- h(T_0,k,\eta) &= - \sum_{l=k \pm 1} l \hat{W}(l) \int_{T_0}^{T_1} \tilde{h}(t,l,lt) (\eta-lt -
    t\text{sgn}(l-k)) \hat{\psi}(\eta-lt) dt,
\end{align*}
and in the following intend to argue by iterated local in time estimates.
Here we first consider $L^2$ estimates and subsequently extend the estimates to
Gevrey regularity.\\

We remark that by Young's convolution inequality the right-hand-side can be
controlled in terms of 
\begin{align*}
 \| (t+\delta l) \hat{W}(l) \hat{\psi}(\eta-lt)\|_{L^1((T_0,T_1)} \leq 100 \hat{W}(l) \|\hat{\psi}(\eta-s)\|_{L^1 (lT_0,lT_1)}
\end{align*}
For $l$ large we may use the decay of $W(l)$ for $l \rightarrow \infty$ and that
thus
\begin{align*}
  100 \hat{W}(l) \|\hat{\psi}\|_{L^1(\R)}\ll 1
\end{align*}
even when we integrate over the whole space.
For $l$ smaller (of which there are only finitely many) we instead may choose the time step
$T_1-T_0$ sufficiently small such that the $L^1$ is small enough to obtain a
contraction. Since the size of the time step is unform, a bound for
$\|h(T)\|_{L^2}$ thus follows by iterating contraction mapping estimates.\\ 

In order to pass to a bound in Gevrey regularity we multiply both sides by
\begin{align*}
  \exp(C \sqrt[3]{\eta})= \exp(C (\sqrt[3]{\eta}-\sqrt[3]{lt}+\sqrt[3]{lt}).
\end{align*}
We thus have to modify our contraction argument so that the contribution due to
\begin{align*}
  \exp(C (\sqrt[3]{\eta}-\sqrt[3]{lt}) \tilde{\psi}(\eta-lt)
\end{align*}
is small.
Using the fact that $|\eta-lt|\leq \delta$ and the Taylor expansion, the exponent
is approximately of size $\exp(C\frac{1}{3}\eta^{-2/3} (\eta-lt)) \leq
\exp(C\frac{1}{3}\eta^{-2/3} \delta) $ and thus small for $\eta$ large (e.g.
larger than $1000$).
For $\eta$ not large we may instead use the unweighted estimate with small but
uniform time steps (to account for $\exp(C 10)$).
The desired estimate thus follows by iterating the contraction mapping bounds.
\end{proof}

Combining the results of Lemma \ref{lem:local} and Theorem \ref{thm:stability}
we thus have seen that the linearized Vlasov-Poisson equations around
self-similar states \eqref{eq:LVP}
\begin{itemize}
\item are \emph{stable} in the Gevrey class $\mathcal{G}_3$ with large constant. 
\item exhibit \emph{norm inflation} due full \emph{echo chains} just like the
  nonlinear problem.
\item exhibit \emph{blow-up} in any regularity class weaker than Gevrey 3. 
\item there exists Gevrey regular initial data which exhibits \emph{blow-up} and
  \emph{physical Landau damping} in the sense \ref{item:physical} at the same time!
\end{itemize}
We thus suggest that the Gevrey norm inflation of \cite{Villani_long, bedrossian2016nonlinear} and plasma echoes \cite{malmberg1968plasma} are
not nonlinear but \emph{secondary linear phenomena}. Furthermore, the physical phenomenon
of Landau damping, which is the asymptotic decay of the force field \ref{item:physical}, is more
\emph{robust} than the scattering to free transport dynamics \ref{item:strong} studied in
\cite{Villani_long, bedrossian2016nonlinear}.
Here, the linear dynamics around self-similar solutions seems to be related to
the modified asymptotic/scattering dynamics, but the precise behavior of the
nonlinear Vlasov-Poisson equations in lower regularity and its effect on Landau
damping remain challenging problems for future research.

\bibliographystyle{alpha} \bibliography{citations2.bib}

\end{document}